\documentclass[a4paper,10pt]{article}

\usepackage[dvipsnames]{xcolor}
\usepackage[sans]{dsfont}
\usepackage[mathscr]{euscript}
\usepackage[centertags]{amsmath}
\usepackage{cite}
\usepackage{amsfonts}
\usepackage{amsthm,enumerate,amssymb}
\usepackage{graphicx}

\usepackage{hyperref}

\newcommand{\numberset}{\mathbb}
\newcommand{\R}{\numberset{R}}
\newcommand{\N}{\numberset{N}}
\newcommand{\E}{\mathbb{E}}
\newcommand{\Prob}{\mathbb{P}}

\newcommand{\C}{\bold C}

\newtheorem{thm}{Theorem}[section]

\newtheorem{cor}[thm]{Corollary}
\newtheorem{lem}[thm]{Lemma}
\newtheorem{prop}[thm]{Proposition}

\newtheorem{defin}[thm]{Definition}
\newtheorem{assu}[thm]{Assumption}
\newtheorem{rem}[thm]{Remark}

\begin{document}

\title{Least singular value and condition number \\ of a square random matrix with i.i.d.\ rows}

\author{M. Gregoratti$^{1}$, D. Maran$^{1}$
\\
\\
$^1$ \textsl{Politecnico di Milano, Dipartimento di Matematica}, \\
\textsl{Piazza Leonardo da Vinci 32, I-20133 Milano, Italy}}

\maketitle

\begin{abstract}
We consider a square random matrix made by i.i.d.\ rows with any distribution and prove that, for any given dimension, the probability for the least singular value to be in $[0,\epsilon)$ is at least of order $\epsilon$. This allows us to generalize a result about the expectation of the condition number that was proved in the case of centered gaussian i.i.d.\ entries: such an expectation is always infinite. Moreover, we get some additional results for some well-known random matrix ensembles, in particular for the isotropic log-concave case, which is proved to have the best behaving in terms of the well conditioning.
\bigskip

\end{abstract}

\section{Introduction}

The first important results about the least singular value $\sigma_\text{min}(\widetilde X)$ and the condition number $\kappa(\widetilde X)$ of a square $n\times n$ random matrix $\widetilde X$ were obtained in 1988. Edelman in \cite{Ede} computed the exact distribution of $\sigma_\text{min}(\widetilde X)$ for a matrix of i.i.d.\ complex standard gaussian entries and the limiting distribution in the i.i.d.\ real standard gaussian case. Kostlan in \cite{Kos} 
proved that $\E[\kappa(\widetilde X)]=+\infty$ whenever the entries are i.i.d.\ real centered gaussian, regardless of the matrix dimension. Two years later Szarek in \cite{Sza} found lower and upper bounds both for $\E[\log\kappa(\widetilde X)]$ and for $\E[\kappa(\widetilde X)^\alpha]$, $0<\alpha<1$, which hold every time the entries are i.i.d.\ standard gaussian, and which depend only on the matrix dimension $n$, the choice of the $p$-norm on $\R^n$, and the choice of $\alpha$.

After twenty years Tao and Vu discovered that, always in the case of i.i.d.\ random entries, the limiting distribution of $\sigma_\text{min}(\widetilde X)$ is universal \cite{Tao}: if the entries moments are bounded, then the cumulative distribution function of the least singular value\footnote{actually, the cumulative distribution function of $n\sigma_\text{min}^2$, since $\sigma_\text{min}\to 0$ as the dimension of the matrix grows} converges (uniformly) to the one of the gaussian case when the dimension $n$ of the matrix grows.

In recent years it were studied some more general classes of random matrices: on one side some works were focused on removing the assumption of gaussian entries and substituting it with a bound on their tail distribution \cite{Rude} or even just with the fact that they admit variance \cite{Reb}; on the other side, some works relaxed the assumption that the entries of the same row are independent and focused on matrices with i.i.d. rows of specific distributions \cite{Adam} \cite{Tik}.

These works where focused on finding upper bounds for the cumulative distribution function of the least singular value $\sigma_\text{min}(\widetilde X)$ as well as lower bounds for the one of the condition number $\kappa(\widetilde X)$ (since the least singular value is smaller the closer the matrix is to singularity while the condition number follows the opposite path, their estimations are usually linked and the present paper does not make an exception).

The articles \cite{Adam},  \cite{Reb}, \cite{Rude} and \cite{Tik} found, under different assumptions, estimations for the asymptotic case which bound the cumulative distrubution up to an error of exponential order in the matrix dimension $n$. Morevorer, two of them,  \cite{Adam} and \cite{Tik} managed to prove even bounds of the type
$$\Prob\Big(\sigma_\text{min}(\widetilde X)\le \epsilon\Big)<f(n,\epsilon)$$
which hold for fixed values of $n$ and $\epsilon>0$.

Then a natural aim could be to find how further these estimations can arrive. Note that, at least for i.i.d.\ real gaussian entries, the above mentioned papers by Edelman, Kostlan and Szarek entail that we cannot find bounds of the type
$$\Prob\Big(\sigma_\text{min}(\widetilde X)\le \epsilon\Big)<f(n)\epsilon^{1-\delta}$$
for any $\delta>0$.
As we are going to show, this is not a characteristic of the gaussian case.

Indeed, we can prove a lower bound for the cumulative distribution function of the least singular value $\sigma_\text{min}(\widetilde X)$ of a square random matrix, of every fixed dimension $n$, in the general setting of i.i.d.\ rows. We do not ask any additional assumption on the rows distribution, that may have unbounded moments or even not admit neither a continuous density function nor a discrete one.

Under these only assumptions we can prove our main results:
\begin{itemize}
\item $\displaystyle\liminf_{\epsilon\to0^+}\frac{\Prob\Big(\sigma_\text{min}(\widetilde X)<\epsilon\Big)}{\epsilon}>0$,
\item $\displaystyle\E \biggl[\frac{1}{\sigma_\text{min}(\widetilde X)}\biggr]=\E\Big[\|\widetilde X^{-1}\|\Big]=+\infty$,
\item $\displaystyle\E\Big[\kappa(\widetilde X)\Big]=\E\Big[\|\widetilde X\|\|\widetilde X^{-1}\|\Big]=+\infty.$
\end{itemize}
The first item generalizes the behaviour of the least singular value of i.i.d.\ real gaussian entries.
The last item generalizes the result by Kostlan on the average condition number. Of course, $\|\cdot\|$ can be any matrix norm and the results are still valid for matrices with i.i.d.\ columns instead of rows.

Moreover, in the cases of a random matrix described by \cite{Adam,Tik}, we get additional results by combining our lower bound with their upper bounds. We prove that the probability of $\sigma_\text{min}(\widetilde X)\in[0,\epsilon)$ grows linearly with $\epsilon$ in a neighbourhood of $0$, as well as we prove an interesting property of the moments of the condition number showing that the isotropic log-concave distribution has the best behaving in terms of the well conditioning.

Of course, our results are trivial if $\Prob\Big(\sigma_\text{min}(\widetilde X)=0\Big)>0$. In particular, our results are trivial in the discrete case, which was vastly studied by \cite{Bou} \cite{Tato} \cite{TaoVu} \cite{TatoVu}. Indeed, if $\widetilde X$ is a square random matrix with i.i.d.\ rows $X_1,\ldots,X_n$ assuming some value $x$ with positive probability, then
$$\Prob\Big(\sigma_\text{min}(\widetilde X)=0\Big)\ge \Prob\Big(X_1=X_2\Big)\ge\Prob\Big(X_1=x\Big)^2>0.$$

It is also easy to see that relaxing our only hypothesis, for example taking shifted random matrices (matrices which are made by the sum of a random matrix with independent entries and a deterministic one), our results may not hold true. Indeed, if
$$\widetilde X=3I+\begin{pmatrix} M_{11} \ M_{12}\\  M_{21} \ M_{22}\end{pmatrix}$$
where $M_{11}, M_{12},  M_{21}, M_{22}$ are i.i.d.\ and such that $M_{11}\in (-1,1)$ a.s., then $\sigma_\text{min}(\widetilde X)>1$ a.s..
This type of matrices has been studied in \cite{San} (gaussian case), \cite{TaoVuv} and \cite{Tik}.

One can also verify that our results may not hold true in the case of an inhomogeneous random matrix (where the entries are independent but not identically distributed), where \cite{Liv} recently discovered some upper bounds for the cumulative distribution function of $\sigma_\text{min}(\widetilde X)$ which generalize the ones of \cite{Rude}.

Therefore, it remains as new open question to find the minimal hypothesis on the random matrix $\widetilde X$ such that our results hold.

As it will be clear in section \ref{nuple}, our tecniques are ineffective in the case of rectangular random matrices, where some estimations for the distribution of the least singular values have been found in \cite{Lit}, \cite{LitRiv} and \cite{Rud}.

\section{Notations}

Given a vector $x\in \R^n$ and a square matrix $A\in \R^{n\times n}$, we introduce the usual vector and operator $p$-norms,  $p\in \N \cup \{+ \infty\}$,
$$\|x\|_p=\sqrt[p]{\sum_{i=1}^n |x(i)|^p},\qquad \forall p\in \N,\qquad\qquad\|x\|_\infty=\max_{i=1,\ldots,n}|x(i)|,$$
$$\|A\|_p=\max_{\|x\|_p=1}\|Ax\|_p.$$
In particular, if we denote the rows of the matrix $A$ by $A_1,\ldots,A_n$, we also have
$$\|A\|_\infty=\max_{i=1,\ldots,n} \|A_i\|_1.$$
Moreover, if we denote by $\sigma_\text{min}(A)$ and $\sigma_\text{max}(A)$ the smallest and the largest \emph{singular value} of $A$ respectively, that is the square root of the smallest and the largest eigenvalue of $A^TA$, then we have
$$\sigma_\text{max}(A)=\|A\|_2=\max_{\|x\|_2=1}\|Ax\|_2,\qquad\qquad\sigma_\text{min}(A)=\min_{\|x\|_2=1}\|Ax\|_2$$
and, if $A$ is invertible,
$$\sigma_\text{min}(A)=\frac{1}{\|A^{-1}\|_2}.$$
Finally, the \emph{condition number} of $A$ in matrix norm $\|\cdot\|$ on $\R^{n\times n}$ is
\[
\mathcal \kappa(A)=
\begin{cases}
\|A\|\,\|A^{-1}\|, & \text{if A is invertible,} \\
+\infty,  & \text{otherwise.}
\end{cases}
\]
The condition number depends on the choice of the matrix norm, but different condition numbers are always pairwise equivalent thanks to the pairwise equivalence of the norms.

\section{The Moulds}\label{moulds}

Our results are based on the introduction of moulds, whose definition is motivated by the following lemma about the expectation of a positive random variable.

\begin{lem}\label{attesa} Let $W$ be a positive random variable such that
$$\liminf_{t\to +\infty} \Big(1-\Prob (W\le t)\Big)t=q>0.$$
Then $\E[W]=+\infty$.
\end{lem}
\begin{proof}
By assumption, there exists $T>0$ such that
$$\Big(1-\Prob (W\le t)\Big)\ge \frac{q}{2t},\qquad \forall t>T,$$
otherwise we could find a sequence $t_k\to \infty$ such that $\displaystyle\lim_{k\to +\infty} \Big(1-\Prob (W\le t_k)\Big)t_k< q/2<q$. Then
$$\E[W]=\int_0^\infty \Prob(W>t)dt\ge \int_T^\infty \Prob(W>t)dt\ \ge \int_T^\infty \frac{q}{2t} dt=\infty.$$
\end{proof}

Motivated by this lemma, we introduce our main definition.

\begin{defin} Let $X$ be a random vector in $\R^n$. For every integer number $m\geq0$, the \emph{$m$-dimensional mould} of $X$, denoted by $\C_m (X)$, is the set of all $x\in\R^n$ such that
$$\liminf_{\epsilon \to 0^+} \frac{\Prob\Big(\|X-x\|_2<\epsilon\Big)}{\epsilon ^m}>0.$$
\end{defin}

Of course, every mould $\C_m(X)$ only depends on the distribution of the random vector and, moreover, it does not change if we replace the euclidean norm in the definition with any other one. Then we can immediately prove the following important feature of the moulds.

\begin{thm}\label{CalP}
Let $X$ be a random vector in $\R^n$ and let $x$ be a point in $\C_m(X)$, $m\geq1$. Then
$$\E\biggl[\frac{1}{\|X-x\|^m}\biggr]=+\infty.$$
\end{thm}

\begin{proof}
By definition of mould, we know that
$$
\liminf_{\epsilon\to 0^+} \frac{\Prob(\|X-x\|<\epsilon)}{\epsilon^m}>0.
$$
Then, after the change of variable $\epsilon=\sqrt[m]{1/t}$, we have
$$
\liminf_{t\to +\infty}  \Prob \biggl(\frac{1}{\|X-x\|^m}>t\biggl)t>0 \implies \liminf_{t\to +\infty}  \biggl(1-\Prob \biggl(\frac{1}{\|X-x\|^m}\le t \biggr)\biggl)t>0.
$$
Finally, thanks to previous lemma \ref{attesa}, this is enough to get
$$
\E\Big[\|X-x\|^{-m}\Big]=+\infty.
$$
\end{proof}

In order to usefully apply such a theorem, we need to explore some other features of the moulds. First of all, moulds are a sequence of sets that obviously grows with the index:

\begin{equation}\label{inscatola}\C_\ell(X)\subseteq \C_m(X), \qquad\forall \ell\leq m.\end{equation}

Moreover, in order to compute the liminf in the definition of moulds, it is enough to compute the liminf along the sequence $\epsilon_k=1/k$.

\begin{prop}\label{calconsucc}
Let $X$ be a random vector in $\R^n$, $x$ be a point in $\R^n$, $m\geq0$. Then
$$
\liminf_{\epsilon\to 0^+} \frac{\Prob\Big(\|X-x\|_2<\epsilon\Big)}{\epsilon^m}=\liminf_{k\to\infty} \frac{\Prob\Big(\|X-x\|_2<1/k\Big)}{(1/k)^m}.
$$
\end{prop}

\begin{proof}
For a given $m$, if we set
$$
f_m(x)=\liminf_{k\to\infty} \frac{\Prob\Big(\|X-x\|_2<1/k\Big)}{(1/k)^m},
$$
then it is enough to show that the liminf computed along any another sequence $\epsilon_j\downarrow0$ has to be bigger or equal to $f_m(x)$.

Thus, given $\epsilon_j\downarrow0$, let us consider the integer part of $1/\epsilon_j$,
$$
k_j=\left[\frac{1}{\epsilon_j}\right],
$$
so that $k_j\uparrow\infty$ and, eventually, $k_j\in\N$ and
$$
\frac{1}{k_j+1}<\epsilon_j\leq\frac{1}{k_j}.
$$
Then
$$
\liminf_{j\to\infty} \frac{\Prob\Big(\|X-x\|_2<\epsilon_j\Big)}{\epsilon_j^m}
\geq\liminf_{j\to\infty} \frac{\Prob\left(\|X-x\|_2<\frac{1}{k_j+1}\right)}{\left(\frac{1}{k_j+1}\right)^m}\,\frac{k_j^m}{(k_j+1)^m}
\geq f_m(x).
$$
\end{proof}

Thus every $m$-dimensional mould $\C_m(X)$ is a borelian subset of $\R^n$, but in general it could be empty. Anyway an important result holds for $m=n$.

\begin{thm}\label{CalFond}
Let $X$ be a random vector in $\R^n$. Then $\Prob\Big(X\in \C_n(X)\Big)=1$.
\end{thm}

\begin{proof}
In order to prove the theorem, we can prove that, if a compact set $K$ occurs with positive probability, i.e.\ $\Prob\Big(X\in K\Big)>0$,
then $K$ contains at least one point $x$ from the mould $\C_n(X)$.

Indeed, this immediately would imply that $\Prob\Big(X\in K\Big)=0$ for every compact set $K\subseteq C_n(X)^\text{c}$, and, by the properties of a probability measure on the borel sets of a metric space,
$$\Prob\Big(X\in B\Big)=\sup_{\substack{K\subseteq B\\K \text{compact}}}\Prob\Big(X\in K\Big)=0$$
for every borelian $B\subseteq \C_n(X)^\text{c}$, and hence the thesis of the theorem for $B=\C_n(X)^\text{c}$.

So, let $K$ be a compact set such that $\Prob\Big(X\in K\Big)=p>0$.

Taken a closed ball $C_0$ containing $K$, closed ball with radius $R$ in infinity norm (namely, an $\R^n$ hypercube), let $\{c_1,c_2,c_3...c_i...c_{2^n}\}$ be the cover of $C_0$ obtained by splitting $C_0$ into $2^n$ identical closed hypercubes (each one of them with radius $R/2$). Then, by sub-additivity, there exists $i$ such that
$$\Prob\Big(X\in K\cap c_i\Big) \ge \frac{p}{2^n}.$$
Let us call $C_1$ the hypercube $c_i$ with this property, which obviously implies $K\cap C_1\neq\emptyset$.

Since $C_1$ is a compact hypercube too, we can iterate this process in order to find a sequence of compact sets $C_j$ such that
\begin{itemize}
\item $C_j\supset C_\ell$ for every $j<\ell$,
\item $\operatorname{radius}_\infty(C_j)=R/2^j$,
\item $K\cap C_j\neq\emptyset$
\item $\displaystyle\Prob\Big(X \in K\cap C_j\Big)\ge \frac{p}{2^{jn}}.$
\end{itemize}
Now, the Axiom of Choice allows us to find a sequence $\{a_j\}_j\subset \R^n$, with $a_j\in C_j\cap K$. Furthermore
\begin{itemize}
\item[(i)] $a_j$ is a Cauchy sequence: $\|a_j-a_\ell\|_\infty<R/2^{j-1}$ for all $\ell\geq j$,
\item[(ii)] $a_\ell$ belongs to $K\cap C_j$ for all $\ell\geq j$,
\item[(iii)] $a_j\to a$, where $a$ belongs to $K\cap C_j$ for all $j$,
\item[(iv)] $C_{j+1}\subseteq \Big\{x:\|x-a\|_\infty<R/2^j\Big\}$ for all $j$.
\end{itemize}
Thus we have found a point $a$ which belongs to $K$ and such that, for every $j$,
$$
\Prob\left(\|X-a\|_\infty\leq\frac{R}{2^j}\right)\geq\Prob\Big(X\in C_{j+1}\Big)\geq\Prob\Big(X\in C_{j+1}\cap K\Big)\geq\frac{p}{2^{(j+1)n}}.
$$

Finally, thanks to this inequality, we can conclude the proof by showing that $a$ belongs also to the mould $C_n(X)$. Indeed, given $0<\epsilon<R/2$, if we consider the integer part of $\log_2(R/\epsilon)$,
$$
j(\epsilon)=\left[\log_2\frac{R}{\epsilon}\right] \in \N,
$$
then we have
$$
\frac{R}{2^{j(\epsilon)+1}}<\epsilon\leq\frac{R}{2^{j(\epsilon)}},
$$
and therefore
$$
\Prob\Big(\|X-a\|_\infty<\epsilon\Big)\geq\Prob\left(\|X-a\|_\infty<\frac{R}{2^{j(\epsilon)+1}}\right)\geq\frac{p}{2^{(j(\epsilon)+2)n}}\geq\frac{p\,\epsilon^n}{4^n\,R^n}.
$$
This implies
$$
\liminf_{\epsilon \to 0^+} \frac{\Prob(\|X-a\|_\infty<\epsilon)}{\epsilon^n}\geq\frac{p}{4^n\,R^n}>0.
$$
\end{proof}

Thus, every random vector $X$ in $\R^n$ takes values almost surely in its $n$-dimensional mould $\C_n(X)$. In particular $\C_n(X)$ cannot be empty. Depending on the distribution of $X$, such a property can be extended also to lower $m$-dimensional moulds $\C_m(X)$.

\begin{prop}\label{dilata}
Let $X$ be random vector in $\R^n$ such that $X\in B$ a.s., $B$ being a borelian subset of $\R^n$. Suppose that there exists a measurable function $d: B\to\R^m$ and a number $c>0$ such that
$$
\|d(x)-d(y)\| \ge c\,\|x-y\|,\qquad\qquad\forall x,y\in B.
$$
Then $\Prob\Big(X\in \C_{m}(X)\Big)=1.$
\end{prop}

Of course, the norms in the theorem do not count.

\begin{proof}
By applying theorem \ref{CalFond} to the random vector $d(X)$, we immediately get
$$
1=\Prob\Big(d(X)\in \C_m(d(X))\Big)=\Prob\Big(X\in  d^{-1}\big(\C_m(d(X))\big)\Big).
$$
So we only need to prove that $ d^{-1}\big(\C_m(d(X))\big)\subseteq \C_m(X)$ in order to get the desired result.
By hypotesis, for every $\epsilon>0$ and for every $x\in \R^n$ we have
$$
\Big(\|X-x\|<\epsilon\Big)\supseteq \Big(\|d(X)-d(x)\|<c\epsilon\Big).
$$
Then, taking any $x\in d^{-1}\big(\C_m(d(X))\big)$, we have
$$
\liminf_{\epsilon\to 0} \frac{\Prob\Big(\|X-x\|<\epsilon\Big)}{\epsilon^m}\ge \liminf_{\epsilon\to 0} \frac{\Prob\Big(\|d(X)-d(x)\|<c\epsilon\Big)}{\epsilon^m}>0,
$$
so that $x\in \C_m(X)$. This shows that $ d^{-1}\big(\C_m(d(X))\big)\subseteq \C_m(X)$ and completes the proof.
\end{proof}

For example, proposition \ref{dilata} immediately implies that $\Prob\Big(X\in \C_{m}(X)\Big)=1$ if $X$ takes values almost surely in some $m$-dimensional linear subspace of $\R^n$.

\section{$n$ i.i.d.\ $n$-dimensional random vectors}\label{nuple}

In order to prove our results about the least singular value and the condition number of a square random matrix, first we have to introduce a peculiar property of an $n$-uple of i.i.d.\ $n$-dimensional random vectors satisfying the following assumption. It is crucial that the number of vectors coincides with the dimension of the space, that is the reason why our results do not extend to rectangular matrices.

\begin{assu}\label{assumption}
We say that $X_1,\ldots,X_n$ satisfy assumption \ref{assumption} if they are i.i.d.\ random vectors in $\R^n$ such that $X_1,\ldots,X_{n-1}$ are linearly independent a.s.\ ($n\geq2$).
\end{assu}

For example, assumption \ref{assumption} is satisfied by $n$ i.i.d.\ random vectors with an absolutely continuous distribution in $\R^n$.

In order to state the peculiar property holding under this assumption, we need, for $n\geq2$, the generalized cross product of $n-1$ vectors in $\R^n$, that is $\wedge: \R^{(n-1)\times n}\to \R^n$,
$$
\wedge(x_1,\ldots,x_{n-1})=
\det\begin{bmatrix}
\bold e_1 & \cdots & \bold e_n\\
x_1(1) & \cdots & x_1(n)\\
\vdots & \cdots & \vdots \\
x_{n-1}(1) & \cdots & x_{n-1}(n)
\end{bmatrix}
$$
where $\bold e_i$ is the $i$-th element of the canonical basis of $\R^n$.
Its properties generalize the features of the $\R^3$ cross product:
\begin{itemize}
\item[(i)] $\wedge(x_1,\ldots,x_{n-1})$ is orthogonal to the vector space spanned by $x_1,\ldots,x_{n-1}$,
\item[(ii)] $\| \wedge(x_1,\ldots,x_{n-1})\|=0\iff x_1,\ldots,x_{n-1}$ are linearly dependent.
\end{itemize}

Finally we can state the above mentioned property, the main result of this section.

\begin{thm}\label{Ygod}
Let $X_1,\ldots,X_n$ be random vectors satisfying assumption \ref{assumption}. Let
$$Y=\frac{\wedge(X_1,\ldots,X_{n-1})}{\|\wedge(X_1,\ldots,X_{n-1})\|_\infty}.$$
Then
$$
0\in \C_1(X_n\cdot Y),\qquad\qquad\E\left[\frac{1}{|X_n\cdot Y|}\right]=+\infty.
$$
\end{thm}

The proof of theorem \ref{Ygod} takes the whole section and, of course, it relays on the introduction of moulds and their basic properties.

First of all, let us remark that $\| \wedge(X_1,\ldots,X_{n-1})\| \neq0$ a.s.\ because of assumption \ref{assumption} and so the random vector $Y$ is well defined. Furthermore, the vector $\wedge(X_1,\ldots,X_{n-1})$ is a.s.\ orthogonal to $X_j$ for all $j=1,\ldots,n-1$, and it is stochastically independent of $X_n$.

\begin{rem}
The vector $Y$ introduced in theorem \ref{Ygod} is similar to the ones introduced in \cite{Liv}(thm 1.2), \cite{Tat} (pagg 6-7) and \cite{Adam} (proof of proposition 2.10). In these three cases it is indicated as the vector orthogonal to the hyperplane spanned by a set of $n-1$ rows and it is normalized with respect to the euclidean norm instead of the infinity norm.
In particular, \cite{Adam} manages to arrive to complementary enstimations to the ours, namely, in that article is proved that (for the isotropic log-concave ensemble, see \ref{LogC})\footnote{In its case the constant depends on the dimension of the matrix and it is universal for every isotropic log-concave distribution while in our case it is different for every random matrix considered.}
$$\Prob(|X_n\cdot Y|<\epsilon)<C\epsilon,$$
while we are proving that
$$\Prob(|X_n\cdot Y|<\epsilon)>c\epsilon$$
for any distribution of $X_n$ and positive $\epsilon$ sufficiently small.
\end{rem}

We begin with the following property of the $(n-1)$-dimensional mould of the random vector $Y$.

\begin{prop}\label{Ycal}
Let $X_1,\ldots,X_n$ be random vectors satisfying assumption \ref{assumption}. Let
$$Y=\frac{\wedge(X_1,\ldots,X_{n-1})}{\|\wedge(X_1,\ldots,X_{n-1})\|_\infty}.$$
Then
$$
Y\in \C_{n-1}(Y) \text{ a.s..}
$$
\end{prop}

\begin{proof}
By construction, the random vector $Y$ belongs to $S_\infty^{n-1}=\Big\{v \in \R^n\;:\; \|v\|_\infty=1\Big\}$ a.s..

Since there exists a measurable dilation $d:S_\infty^{n-1}\to \R^{n-1}$, the thesis follows immediately by Lemma \ref{dilata}.
\end{proof}

The next step is to study the special case of \emph{bounded} random vectors $X_1,\ldots,X_n$, where we can prove the desired results by showing a link between the $n-1$ dimensional mould of $Y$ and the properties of $X_n$.

\begin{prop}\label{Ygodbd}
Let $X_1,\ldots,X_n$ be random vectors satisfying assumption \ref{assumption} and, moreover, let them be bounded. Let
$$Y=\frac{\wedge(X_1,\ldots,X_{n-1})}{\|\wedge(X_1,\ldots,X_{n-1})\|_\infty}.$$
Then
\begin{enumerate}
\item $y\in \C_{n-1}(Y)\quad\implies\quad 0\in \C_1(X_n\cdot y)$,
\item $0\in \C_1(X_n\cdot Y)$,
\item $\displaystyle\E\left[\frac{1}{|X_n\cdot Y|}\right]=+\infty$.
\end{enumerate}
\end{prop}

\begin{proof}
We prove the proposition thesis by thesis.
\begin{enumerate}
\item
Since $X_1,\ldots,X_n$ are i.i.d., for every $y\in\R^n$ and for every $\epsilon>0$ we have
$$\Prob\Big(|X_n\cdot y|<\epsilon\Big)=\sqrt[n-1]{\Prob\left(\bigcap_{j=1}^{n-1}\Big(|X_j\cdot y|<\epsilon\Big)\right)}.$$
Now, let us take $r>0$ such that $\|X_j\|_1<r$ a.s., and let us denote by $\widehat X$ the $\R^{(n-1)\times n}$ random matrix with rows  $X_j: 1\le j\le n-1$.

Then we have the following relationships among events
\begin{multline*}
\bigcap_{j=1}^{n-1}\Big(|X_j\cdot y|<\epsilon\Big)=\Big(\|\widehat X y\|_\infty<\epsilon\Big)
=\Big(\|\widehat X (y-Y)\|_\infty<\epsilon\Big)\supseteq \Big(\|\widehat X\|_\infty \|y-Y\|_\infty<\epsilon\Big)\\
\supseteq \Big(r \|y-Y\|_\infty<\epsilon\Big),
\end{multline*}
so that
$$\liminf_{\epsilon\to 0^+} \frac{\Prob\Big(|X_n\cdot y|<\epsilon\Big)}{\epsilon}\ge \liminf_{\epsilon\to 0^+} \sqrt[n-1]{\frac{\Prob\Big(\|Y-y\|_\infty<\epsilon/r\Big)}{\epsilon^{n-1}}}.$$
Therefore $0\in \C_1(X_n\cdot y)$ for every $y\in \C_{n-1}(Y)$.
\item
Let us consider the following measurable functions of $y\in\R^n$
$$
\phi_k(y)=\Prob\Big(|X_n\cdot y|<1/k\Big),\quad k\in\N,\qquad\qquad f(0|y)=\liminf_{k\to\infty}k\,\phi_k(y).
$$
Then, by the previous point and by proposition \ref{calconsucc}, for every $y\in \C_{n-1}(Y)$ we have $f(0|y)>0$ so that there exists a $k(y)\in\N$ such that $k\,\phi_k(y)\geq\frac{1}{2}\,f(0|y)>0$ for any $k\geq k(y)$.

Thus, if we consider
$$
B_m=\left\{y\in \C_{n-1}(Y)\;:\; k\,\phi_k(y)\geq\frac{1}{m}\quad\forall k\geq m\right\},\qquad m\in\N,
$$
we get a sequence of borel sets in $\R^n$ growing to $\C_{n-1}(Y)$. Indeed, for every $m\geq1$ we have $B_m\subseteq B_{m+1}\subseteq \cup_\ell B_\ell \subseteq \C_{n-1}(Y)$, obviously, but we also have the opposite inclusion $\C_{n-1}(Y)\subseteq\cup_\ell B_\ell$ because, taken any $y\in \C_{n-1}(Y)$, there exists $m\in\N$ such that $k\,\phi_k(y)\geq\frac{1}{m}$ for every $k\geq k(y)$, that is $y\in B_m$.

By monotonicity, this implies that $\Prob\Big(Y\in B_m\Big)\to\Prob\Big(Y\in \C_{n-1}(Y)\Big)$, which equals 1 by proposition \ref{Ycal}, so that there exists $m_\star$ such that $\Prob\Big(Y\in B_{m_\star}\Big)\geq1/2$.

At this point, using the basic properties of conditional expectation, we have
$$\Prob\Big(|X_n\cdot Y|<1/k\Big)=\E\Big[\E\Big[I_{[0,1/k)}(|X_n\cdot Y|)\Big|Y\Big]\Big]$$
and, thanks to the freezing lemma, which we can apply due to the independence of $X_n$ and $Y$,
$$\E\Big[\E\Big[I_{[0,1/k)}(|X_n\cdot Y|)\Big|Y\Big]\Big]=\E\Big[\phi_k(Y)\Big].$$
Then proposition \ref{calconsucc} allows us to conclude:
\begin{multline*}
k\,\Prob\Big(|X_n\cdot Y|<1/k\Big)
=k\,\E\Big[\E\Big[I_{[0,1/k)}(|X_n\cdot Y|)\Big|Y\Big]\Big]=k\,\E\Big[\phi_k(Y)\Big]\\
\geq k\,\E\Big[\phi_k(Y)\,I_{B_{m_\star}}(Y)\Big]\geq\frac{1}{m_\star}\,\Prob\Big(Y\in B_{m_\star}\Big)>0.
\end{multline*}
\item Thesis 3 follows immediately from thesis 2 thanks to theorem \ref{CalP}.
\end{enumerate}
\end{proof}

Finally we can prove theorem \ref{Ygod}.

\begin{proof}[Proof of theorem \ref{Ygod}]
The result is already proved for bounded random vectors thanks to proposition \ref{Ygodbd}. Then, taken a $\rho>0$ such that the event
$$
E_\rho = \bigcap_{i=1}^n\Big(\|X_i\|<\rho\Big)
$$
has positive probability, it is enough to consider the conditional probability
$$
\Prob_\rho(\cdot)=\Prob(\cdot|E_\rho).
$$
Indeed for every $\epsilon>0$
$$
\frac{\Prob\Big(|X_n\cdot Y|<\epsilon\Big)}{\epsilon}\geq \frac{\Prob_\rho\Big(|X_n\cdot Y|<\epsilon\Big)}{\epsilon}\,\Prob(E_\rho),
$$
where the right hand side has a strictly positive liminf as $\epsilon\to0^+$ by proposition \ref{Ygodbd}, as the random vectors $X_1,\ldots,X_n$ are bounded under $\Prob_\rho$ and it is a straightforward verification that they are also $\Prob_\rho$-i.i.d. and still satisfy the assumption \ref{assumption}.

Therefore $0\in \C_1(X_n\cdot Y)$ and the full thesis immediately follows thanks to theorem \ref{CalP}.
\end{proof}

\section{Least singular value $\sigma_\text{min}(\widetilde X)$}
Thanks to the introduction of the definition of moulds for a random vector (section \ref{moulds}) and thanks to the properties deduced for an $n$-uple of i.i.d.\ random vectors in $\R^n$ (section \ref{nuple}), we can finally come to our main results. Let us start with the least singular value.

\subsection{The main result for $\sigma_\text{min}(\widetilde X)$}
\begin{thm}\label{lsv}
Let $\widetilde X$ be a square random matrix with i.i.d.\ rows. Then
$$
0\in \C_1\Big(\sigma_\text{min}(\widetilde X)\Big)\qquad\text{i.e.}\qquad \liminf_{\epsilon\to0^+}\frac{\Prob\Big(\sigma_\text{min}(\widetilde X)<\epsilon\Big)}{\epsilon}>0,
$$
and, if $\widetilde X$ is invertible almost surely,
$$
\E\left[\frac{1}{|\sigma_\text{min}(\widetilde X)|}\right]=\E\Big[\|\widetilde{X}^{-1}\|\Big]=+\infty.
$$
\end{thm}

\begin{proof}
If the random matrix $\widetilde{X}$ is singular with positive probability the thesis is trivial. Otherwise its rows $X_1,\ldots,X_n$ satisfy assumption \ref{assumption} and we can consider the random vector $Y$ of theorem \ref{Ygod}. Then it is enough to observe that, since $\|Y\|_\infty=1$ and so $\|Y\|_2\ge 1$,
$$
\Big(\sigma_\text{min}(\widetilde X)<\epsilon\Big)=\Big(\min_{\|y\|_2=1}\|\widetilde X\,y\|_2<\epsilon\Big)
\supseteq \left(\frac{\|\widetilde X\,Y\|_2}{\|Y\|_2}<\epsilon\right)
\supseteq \Big(\|\widetilde X\,Y\|_2<\epsilon\Big)
=\Big(|X_n\cdot Y|<\epsilon\Big),
$$
to deduce
$$
\liminf_{\epsilon\to0^+}\frac{\Prob\Big(\sigma_\text{min}(\widetilde X)<\epsilon\Big)}{\epsilon}\geq\liminf_{\epsilon\to0^+}\frac{\Prob\Big(|X_n\cdot Y|<\epsilon\Big)}{\epsilon}>0.
$$
The full thesis then follows thanks to theorem \ref{CalP}.
\end{proof}

Since the least singular value is invariant under transposition, the theorem holds for matrices with i.i.d.\ columns, too.

\subsection{Additional results for $\sigma_\text{min}(\widetilde X)$ for some well known ensembles}
After finding a lower bound of $k\epsilon$ for the probability that the least singular value $\sigma_\text{min}$ of a square random matrix with generic i.i.d.\ rows is smaller than $\epsilon$, it is natural to ask if this estimation can be improved for particular random matrix ensembles.

Of course, if $\widetilde X$ is a random matrix with i.i.d.\ discrete rows, $\Prob(\sigma_\text{min}(\widetilde X)=0)>0$ so the previous result \ref{lsv} becomes trivial in this case.

However, there are lots of ensembles where the previous lower bound can be associated to proper upper bounds which together determine the behaviour of the cumulative distribution of $\sigma_\text{min}$ in the neighbourhood of $0$.

\subsubsection{Matrices of i.i.d.\ rows with isotropic log-concave distribution}\label{LogC}

A random vector has a log-concave distribution if for every $\lambda\in (0,1)$, said $f(x)$ its density function, we have
$$f(\lambda x+(1-\lambda)y)\le f(x)^\lambda f(y)^{1-\lambda}.$$
A random vector is said to be isotropic if it has mean value zero.

In \cite{Adam} Adamczak et al.\ show (corollary 2.14)
that if $\widetilde X$ is a square random matrix of dimension $n$ with i.i.d.\ rows drawn from an isotropic log-concave distribution,
$$\forall \epsilon \in(0,1),\ \forall \delta>0,\ \exists C_\delta \quad:\quad \Prob\Big(\sigma_\text{min}(\widetilde X)<n^{-1/2}\epsilon\Big)<\epsilon^{1-\delta}C_\delta.$$

If the matrix is larger than a fixed dimension $n_0$, we can even choose $\delta=0$ in the previous estimation, as it was proved by Tikhomirov in \cite{Tik} (corollary 1.4), obtaining
$$\Prob\Big(\sigma_\text{min}(\widetilde X)<n^{-1/2}t\Big)<Ct,\qquad\forall t>0.$$
The dimension $n_0$ is universal, in the sense that it is independent of the isotropic log-concave distribution, as well as $C$ is a universal constant independent both of the isotropic log-concave distribution and of the dimension $n>n_0$.
Summing up our result and the ones of \cite{Adam} and \cite{Tik} we get the following corollary.

\begin{cor}\label{logC}
Let $\widetilde X$ be a random matrix with i.i.d.\ rows drawn from an isotropic log-concave distribution. Then, for every $\delta>0$ there exist $0<k_1<k_2$ such that
$$k_1 \epsilon<\Prob\Big(\sigma_\text{min}(\widetilde X)<\epsilon\Big)<k_2 \epsilon^{1-\delta}$$
(where $k_2=C_\delta \sqrt{n}$ and $C_\delta$ only depends on $\delta$)
holds for positive $\epsilon$ sufficiently small. Moreover, there exists a universal constant $n_0$ such that, if the size of $\widetilde X$ is greater than $n_0$, then
$$k_1 \epsilon<\Prob\Big(\sigma_\text{min}(\widetilde X)<\epsilon\Big)<k_2 \epsilon$$
(where $k_2=C\sqrt{n}$ and $C$ is a universal constant) holds for positive $\epsilon$ sufficiently small.
\end{cor}

\subsubsection{Matrices of i.i.d.\ $L^2$ contiuous entries (large n)}

Tikhomirov in \cite{Tik} proved (corollary 1.3) that for any $L>0$ there is $v(L)>0$ and $n_0\in\N$ such that for all matrices $\widetilde X$ of dimension $n>n_0$ of i.i.d.\ continuous entries $X_{ij}$ with density $f$ such that
$$\E[X_{ij}]=0,\qquad \E[X_{ij}^2]=1,\qquad \sup_{x\in \R} f(x)<L$$
we have
$$\Prob\Big(\sigma_\text{min}(\widetilde X)<n^{-1/2}t\Big)<v(L)\,t, \qquad\forall t>0.$$

Summing up with \ref{lsv}, we have that even in this case the probability of the least singular value of being small is a first order infinitesimal in the case when the matrix is big enough.
\begin{cor}
Let $\widetilde X$ be an $n \times n$ ($n>n_0$ universal constant) random matrix with i.i.d.\ continuous entries of mean zero and unit variance whose density function is bounded. Then there exist $0<v_1<v_2$ such that
$$v_1 \epsilon<\Prob\Big(\sigma_\text{min}(\widetilde X)<\epsilon\Big)<v_2 \epsilon$$
holds for positive $\epsilon$ sufficiently small.
\end{cor}

\section{Condition number $\kappa(\widetilde X)$}

Last but not least the condition number.

\subsection{The main result for $\kappa(\widetilde X)$}

\begin{thm}\label{condT}
Let $\widetilde X$ be a square random matrix with i.i.d.\ rows. Then, for every choice of the matrix norm,
$$\E\Big[\kappa(\widetilde X)\Big]=+\infty.$$
\end{thm}

\begin{proof}
If the random matrix $\widetilde{X}$ is singular with positive probability the thesis is trivial. Otherwise, when $\widetilde X$ is invertible a.s., it is enough to prove the theorem for the operator norm induced by the norm infinity of $\R^n$, as condition numbers are pairwise equivalent for a change of the matrix norm.

We prove the theorem in two steps, first for rows $X_1,\ldots,X_n$ bounded from below, then for the general case of $\widetilde X$ invertible a.s..
\begin{enumerate}
\item If $\|X_1\|_1>\rho$ a.s.\ for some $\rho>0$, then the thesis immediately follows. Indeed, such a condition gives
$$
\|\widetilde X\|_\infty=\max_i \|X_i\|_1>\rho \text{ a.s.}$$
and so, by theorem \ref{lsv},
$$
\E\Big[\kappa_\infty(\widetilde X)\Big]=\E\Big[\|\widetilde X\|_\infty\,\|\widetilde X^{-1}\|_\infty \Big] > \rho\,\E\Big[\|\widetilde{X}^{-1}\|\Big]=+\infty.$$
\item If $\widetilde X$ is invertible a.s., then $\Prob\Big(\|X_i\|>0\Big)=1$ and, by monotonicity, there exists $\rho>0$ such that $\Prob\Big(\|X_1\|_1>\rho\Big)>0$. Thus, the event
$$
E_\rho=\bigcap_{i=1}^n\Big(\|X_i\|_1>\rho\Big)
$$
has positive probability and we can consider the conditional probability
$$
\Prob_\rho(\cdot)=\Prob(\cdot|E_\rho).
$$
As $\Prob(A)\geq\Prob_\rho(A)\,\Prob(E_\rho)$ for every event $A$, we also have $\E[W]\geq\E_\rho[W]\,\Prob(E_\rho)$ for every random variable $W\geq0$.
Thus $\E\Big[\kappa(\widetilde X)\Big]\geq\E_\rho\Big[\kappa(\widetilde X)\Big]\,\Prob(E_\rho)=+\infty$ by step 1, as the random vectors $X_1,\ldots,X_n$ are bounded from below under $\Prob_\rho$ and it is a straightforward verification that they are also $\Prob_\rho$-i.i.d. and satisy assumption \ref{assumption}.
\end{enumerate}
\end{proof}

This theorem is a generalization of \cite{Kos}, theorem 5.2, in which it was shown that the average condition number for a random matrix with i.i.d. gaussian entries was infinite.

Since the condition number in euclidean norm is invariant under transposition, the previous theorem holds for matrices with i.i.d.\ columns, too.

\subsection{Additional result for $\kappa(\widetilde X)$ in the isotropic log-concave case}

Again, in \cite{Adam} Adamczak et al.\ proved an upper bound for the condition number (corollary 2.15) in the isotropic log-concave case: for every square random matrix $\widetilde X$ with $n$ columns (or rows) i.i.d.\ with isotropic log-concave distribution and for every $\delta>0$, there exists $C_\delta$ such that
$$\Prob\Big(\kappa(\widetilde X)>nt\Big)\le \frac{C_\delta}{t^{1-\delta}},\qquad\forall t>0.$$

This result, which bounds the probability that the condition number is high, can be merged with theorem \ref{condT} to prove the following corollary.

The corollary shows that, under the isotropic log-concave hypothesis, $\alpha=1$ is the least number such that $\E[\kappa(\widetilde X)^\alpha]=+\infty$.

\begin{cor}
Let $\widetilde X$ be a square random matrix with i.i.d.\ rows (or columns) with isotropic log-concave distribution.
Then
$$\E\Big[\kappa(\widetilde X)^\alpha\Big]<+\infty\iff \alpha<1.$$
\label{last}
\end{cor}

\begin{proof}
Our theorem \ref{condT} proves that
$$\alpha \ge 1\implies \E\Big[\kappa(\widetilde X)^\alpha\Big]=+\infty.$$
So it is enough to show that
$$\alpha < 1\implies \E\Big[\kappa(\widetilde X)^\alpha\Big]<+\infty.$$
By the above mentiond result we have
$$\forall \delta>0, \ \exists C_\delta>0 \quad:\quad \Prob\Big(\kappa(\widetilde X)>t\Big)\le \frac{C_\delta}{t^{1-\delta}},\quad\forall t>0,$$
and so it follows that, for all $t>0$ and for all $\alpha \in (0,1)$,
$$\Prob\Big(\kappa(\widetilde X)^\alpha>t\Big)=\Prob\Big(\kappa(\widetilde X)>t^{1/\alpha}\Big)\le \frac{C_\delta}{t^{(1-\delta)/\alpha}}.$$
Now, since $\alpha<1$, we can choose $\delta$ positive such that $$\eta=(1-\delta)/\alpha>1.$$
This means that there exist $\eta>1$ and $C_\delta>0$ such that
$$\Prob\Big(\kappa(\widetilde X)>t\Big)\le \frac{C_\delta}{t^\eta},\quad\forall t>0.$$
Then
$$\E\Big[\kappa(\widetilde X)\Big]=\int_0^\infty \Prob\Big(\kappa(X)>t\Big)\ dt \le \int _0^1 1\ dt+\int_1^\infty \frac{C_\delta}{t^\eta}\ dt=1+\frac{C_\delta}{\eta-1},$$
which is less than infinity since $\eta>1$.
\end{proof}

\section{Final remarks}

This last result shows again that, for random matrices with i.i.d.\ isotropic log-concave rows, our lower bound estimations of the least singular value and of the condition number are complementary to the upper bounds known from the literature: \cite{Adam} and \cite{Tik} give
$$\exists k_\delta,k>0\quad:\quad\Prob\Big(\sigma_\text{min}(\widetilde X)<\epsilon\Big)<
\begin{cases}k_\delta\,\epsilon^{1-\delta},&\forall \delta>0,\;\forall 0<\epsilon<1,\\
k\epsilon,&\forall\epsilon>0, \text{ if the size of } \widetilde X \text{ is large enough,} \end{cases}$$
$$\E\Big[\kappa(\widetilde X)^\alpha\Big]<\infty,\quad\forall \alpha<1,$$
while we proved (corollaries \ref{logC} and \ref{last}) that
$$\exists k_1,\epsilon_0>0\quad:\quad k_1\epsilon<\Prob\Big(\sigma_\text{min}(\widetilde X)<\epsilon\Big)\quad\forall 0<\epsilon<\epsilon_0,$$
$$\alpha<1\iff \E\Big[\kappa(\widetilde X)^\alpha\Big]<\infty.$$
This means that, for every random matrix with i.i.d.\ rows, even if they do not admit a density function or their moments are unbounded, the probability of the least singular value of laying in the interval $[0,\epsilon)$ is at least of the order of $\epsilon$, and in some special cases such as the log-concave ensembles it is exactly of that order. Fortunately, for these distributions we can even bound the previous probability with constants that depends only on the dimension of the matrix and on some universal constants.

Similarly, taking a random matrix where the rows are i.i.d., then inevitably $$\kappa(\widetilde X)\notin L^1.$$
However, choosing the previous particular ensembles we can have a slightly weaker integrability,
$$\kappa(\widetilde X)^\alpha \in L^1, \qquad \forall \alpha<1.$$
As this one is the best achievable integrability, it is shown that the isotropic log-concave distributions are among the "nicest" ones in terms of the well-conditioning of a matrix with those rows.

\end{document}